\documentclass[12pt]{amsart}
\usepackage{amsmath,amssymb,amsthm}
\theoremstyle{plain}
\newtheorem{theorem}{Theorem}[section]

\newtheorem{corollary}[theorem]{Corollary}

\theoremstyle{definition}

\newtheorem{example}[theorem]{Example}

\makeatletter
\@addtoreset{equation}{section}
\makeatother

\makeatletter
\newcommand{\Spvek}[2][r]{%
  \gdef\@VORNE{1}
  \left(\hskip-\arraycolsep%
    \begin{array}{#1}\vekSp@lten{#2}\end{array}%
  \hskip-\arraycolsep\right)}

\def\vekSp@lten#1{\xvekSp@lten#1;vekL@stLine;}
\def\vekL@stLine{vekL@stLine}
\def\xvekSp@lten#1;{\def\temp{#1}%
  \ifx\temp\vekL@stLine
  \else
    \ifnum\@VORNE=1\gdef\@VORNE{0}
    \else\@arraycr\fi%
    #1%
    \expandafter\xvekSp@lten
  \fi}
\makeatother
\begin{document}
\title[Two Measures on Cantor Sets]{Two Measures on Cantor Sets}
\author{G{\"o}kalp Alpan}
\author{Alexander Goncharov}

\begin{abstract}
We give an example of Cantor type set for which its equilibrium measure and the corresponding Hausdorff measure are mutually absolutely continuous. Also we show that these two measures are regular in Stahl-Totik sense.
\end{abstract}

\maketitle
\section{Introduction}

The relation between the  $\alpha$ dimensional Hausdorff measure $\Lambda_\alpha$ and the harmonic measure $\omega$ on a finitely connected domain $\Omega$ is understood well. Due to Makarov  \cite{Makarov}, we know that,
for a simply connected domain, $\dim{\omega}=1$ where $\dim{\omega}:= \inf\{\alpha : \omega \perp \Lambda_\alpha\}$. Pommerenke \cite{Pommerenke}  gives a full characterization of parts of $\partial\Omega$ where $\omega$ is absolutely continuous or singular with respect to a linear Hausdorff measure. Later similar facts were obtained
for finitely connected domains. In the infinitely connected case there are only particular results. Model example
here is $\Omega = \overline{\mathbb{C}} \setminus K$ for a Cantor-type set $K$. On most of such sets we have the strict inequality $\dim{\omega}< \alpha_K $ (see, e.g. \cite{Batakis}, \cite{Makarov Volberg}, \cite{Volberg}, \cite{Zdunik}), where $\alpha_K $ stands for the Hausdorff dimension of $K$.
This inequality implies that $\Lambda_{\alpha_K} \perp \omega$ on $K$. These results motivate the
problem to find a Cantor set for which its harmonic measure and the corresponding Hausdorff measure are not
mutually singular.

Recall that, for a dimension function $h$,  a set $E\subset \mathbb{C}$ is an $h$-set if  $0<\Lambda_h(E) < \infty$ where $\Lambda_h$ is the Hausdorff measure corresponding to the function $h$. We consider  introduced in \cite{Goncharov} Cantor-type sets $K(\gamma).$  In section 2 we present a function $h$ that makes $K(\gamma)$ an $h$-set. In section 3 we show that $\Lambda_h$ and $\omega$ are mutually absolutely continuous for $K(\gamma)$. In the last section we prove that these two measures are regular in Stahl-Totik sense.

We will denote by $\log$ the natural logarithm, $Cap(\cdot)$ stands for the logarithmic capacity.

\section{Dimension function of $K(\gamma)$ }

A function $h:{\mathbb{R}}_+ \to {\mathbb{R}}_+$ is called a dimension function if it is increasing, continuous
and $h(0)=0.$  Given set $E\subset \mathbb{C},$ its $h$-Hausdorff measure is defined as
\begin{equation}
\Lambda_h(E)=\lim_{\delta \rightarrow 0} \inf \left\{ \sum h(r_j) : E \subset \bigcup B(z_j, r_j)\,\,\,
 \mbox {with}\,\,\,   r_j \leq \delta \right\},
\end{equation}
where $B(z, r)$ is the open ball of radius $r$ centered at $z$.

For the convenience of the reader we repeat the relevant material from \cite{Goncharov}.
Given sequence $\gamma = (\gamma_s)_{s=1}^{\infty}$ with $0<\gamma_s\leq \frac{1}{32},$ let $r_0=1$ and $r_s=\gamma_s r_{s-1}^2$ for $s\in \mathbb{N}$. Define $P_2(x)=x(x-1)$ and $P_{2^{s+1}}=P_{2^s}(P_{2^s}+r_s)$ for $s\in \mathbb{N}$.
 Consider the set $$ E_s:=\{x\in {\mathbb{R}}:\,P_{2^{s+1}} (x)\leq 0\} =\cup_{j=1}^{2^s}I_{j,s}.$$
The $s$-th level {\it basic} intervals $I_{j,s}$ with lengths $l_{j,s}$ are disjoint and
$\max_{1\leq j \leq 2^s} l_{j,s} \to 0$ as $s\to \infty.$ Since $E_{s+1} \subset E_s$, we have a Cantor type set $K(\gamma) := \cap_{s=0}^{\infty} E_s.$ The set $K(\gamma)$ is non polar if and only $\sum_{s=1}^{\infty} 2^{-s}\log{\frac{1}{\gamma_s}} < \infty$. In this paper we make the assumption

\begin{equation}
\sum_{s=1}^\infty \gamma_s < \infty.
\end{equation}
Let $M:=1+\exp{\left(16\sum_{s=1}^\infty \gamma_s\right)},$ so $M>2$, and $\delta_s:= \gamma_1\gamma_2\dots\gamma_s$.
By Lemma 6 in \cite{Goncharov}\label{zzz},
\begin{equation}
\delta_s < l_{j,s}< M\cdot\delta_s \,\,\,\mbox{for} \,\,\, 1 \leq j\leq 2^s.
\end{equation}

We construct a dimension function for $K(\gamma),$ following Nevanlinna \cite{Nevanlinna}.
Let  $\eta( \delta_s)=s$ for $s\in \mathbb{Z}_+$ with $\delta_0:= 1.$
We define $\eta(t)$ for $(\delta_{s+1},\delta_{s})$ by
$$
\eta(t)= s+ \frac{\log{\frac{\delta_s}{t}}}{\log{\frac{\delta_s}{\delta_{s+1}}}}.
$$
This makes $\eta$ continuous and monotonically decreasing on $(0,1].$  In addition, we have $\lim_{t\rightarrow 0}\eta(t)=\infty$. Also observe that, for the derivative of $\eta$ on $(\delta_{s+1}, \delta_{s})$, we have
$$
\frac{d\eta}{dt}=\frac{-1}{t\log{\frac{1}{\gamma_{s+1}}}}\geq \frac{-1}{t\log{32}} \,\,\,\,\,\,\mbox {and} \,\,\,\,\,\,
\frac{d\eta}{d\log{t}}\geq \frac{-1}{\log{32}}.
$$

Define $h(t)=2^{-\eta(t)}$ for $0<t\leq 1$ and $h(t)=1$ for $t > 1.$ Then $h$ is  a dimension function
with $h(\delta_s) =2^{-s}$ and
$$ \frac{d\log{h}}{d\log{t}} \leq \frac{\log{2}}{\log{32}} <1. $$

Therefore if $m>1$ and $r\leq 1$ we get the following inequality:
$$ \log{\frac{h(r)}{h\left(\frac{r}{m}\right)}}< \int_{r/m}^r d\log{t} = \log{m}.$$

Finally, we obtain
\begin{equation} \label{pp}
h(r)<m\cdot{h\left(\frac{r}{m}\right)} \,\,\,\, \mbox {for}\,\,\,  m>1 \,\,\,\mbox {and } \,\,\, 0< r\leq 1.
\end{equation}

Let us show that $K(\gamma)$ is an $h$-set for the given function $h.$
\begin{theorem}\label{ygyg} Let $\gamma$ satisfy (2.2).  Then  $ 1/8 \leq\Lambda_h(K(\gamma))\leq M/2$.
\end{theorem}

\begin{proof}
First, observe that, by (2.3), for each $s\in\mathbb{N}$ the set $K(\gamma)$ can be covered by $2^s$ intervals of length $M\cdot{\delta_s}.$  Since $M/2>1,$ we have by (2.4),
$$
\Lambda_h(K(\gamma))\leq \limsup_{s\to \infty} (2^s\cdot h({M/2\cdot{\delta_s}}))\leq \limsup_{s\to \infty} (2^s\cdot{M/2}\cdot{h(\delta_s)})=M/2.
$$

We proceed to show the lower bound. Let $(J_{\nu})$ be an open cover of $K(\gamma)$. Then, by compactness, there are finitely many intervals $(J_{\nu})_{\nu=1}^m$ that cover $K(\gamma)$. Since $K(\gamma)$ is totally disconnected, we can assume that these intervals are disjoint. Each $J_{\nu}$ contains a closed subinterval $J_{\nu}^{\prime}=[a_{\nu}, b_{\nu}]$ whose endpoints belong to $K(\gamma)$ and covers all points of $K(\gamma)$ in $J_{\nu}$. Since the intervals $(J_{\nu}^{\prime})_{\nu=1}^m$ are disjoint, all $a_{\nu}, b_{\nu}$ are endpoints of some basic intervals. Let $n$ be the minimal number such that all $(a_{\nu})_{\nu=1}^m ,
(b_{\nu})_{\nu=1}^m$ are the endpoints of $n-$th level. Thus, each $I_{j,n}$ for $1\leq j \leq 2^n$ is contained in some $J_{\nu}^{\prime}.$  Let $N_{\nu}$ be the number of $n$-th level intervals in $J_{\nu}^{\prime}.$
Clearly,  $\sum_{\nu=1}^{m} N_{\nu} =2^n.$

For a fixed $\nu \in \{1,2,\ldots, m\}$, let $q_{\nu}$ be the smallest number such that $J^{\prime}_\nu$ contains at least one basic interval $I_{j,q_{\nu}}.$ Clearly, $q_{\nu}\leq n$ and $l_{j,q_{\nu}}\leq d_{\nu}$ where $d_{\nu}$ is the length of $J_{\nu}$. Therefore, by (2.3),
$$
h(d_{\nu})\geq h(l_{j, q_{\nu}})\geq h(\delta_{q_{\nu}})=2^{-q_{\nu}}.
$$

Let us cover $J_{\nu}^{\prime}$ by the smallest set $G_{\nu}$ which is a finite union of adjacent intervals of the level $q_{\nu}.$  Observe that $G_{\nu}$ consists of at least one and at most four such intervals. Each interval of
the $q_{\nu}-$th level contains $2^{n-q_{\nu}}$ subintervals of the $n-$th level. This gives at most
$2^{n-q_{\nu}+2}$ intervals of level $n$ in the set $G_{\nu}$. Hence
$$
N_{\nu} \leq 2^{n-q_{\nu}+2}.
$$
Therefore,
$$
\sum_{\nu=1}^{m}h(d_\nu) \geq \sum_{\nu=1}^{m} 2^{-q_{\nu}} \geq 2^{-n-2}\,\sum_{\nu=1}^{m} N_{\nu}= 1/4.
$$
Since $h(d)<2\cdot h(d/2)$ from (2.4), finally we obtain the desired bound.
\end{proof}

 Similar arguments apply to the case of a part of $K(\gamma)$ on any basic interval.

\begin{corollary} \label{eee} Let $\gamma$ satisfy (2.2).  Then
$ 2^{-s-3} \leq\Lambda_h(K(\gamma)\cap I_{j,s})\leq M \cdot 2^{-s-1}$
for each $s\in \mathbb{N}$ and $1\leq j \leq 2^s.$
\end{corollary}

{\sc Remark}. A set $E$ is called {\it dimensional} if there is at least one dimension function $h$ that makes $E$
an $h-$set. It should be noted that not all sets are dimensional. If we replace the condition $h(0)= 0$
by $h(0)\geq 0,$ then any sequence gives a trivial example of a dimensionless set.
Best in \cite{Best} presented an example of a dimensionless Cantor set provided  $h(0)= 0$. The author considered
dimension functions with the additional condition of concavity, but did not used it in his construction.

\section{Harmonic Measure and Hausdorff measure for $K(\gamma)$}
 Suppose we are given a non polar compact set $K$ that coincides with its exterior boundary. Then for the
equilibrium measure $\mu_K$ on $K$  we have the representation
$\mu_K(\cdot)=\omega(\infty, \cdot, \overline{\mathbb{C}}\setminus K)$ in terms of the value of the
harmonic measure at infinity (see e.g. \cite{Ransford}, T.4.3.14).
 Moreover, since measures $\omega(z_1,\cdot,  \overline{\mathbb{C}}\setminus K)$ and $\omega(z_2,\cdot,  \overline{\mathbb{C}}\setminus K)$ are mutually absolutely continuous (see e.g. \cite{Ransford} Cor.4.3.5), our main result is valid even if, instead of $\mu_{K(\gamma)},$ we take the measure corresponding to the value of
the harmonic measure at any other point.

The next theorem follows immediately from the definition of $\Lambda_h$. It is a simple part of
Frostman's theorem (see e.g. T.D.1 in \cite{Garnett}).

\begin{theorem}  Let $h$ be a dimension function. If $\mu$ is a positive Borel measure such that
$$\mu(B(z,r))\leq h(r)$$ for all $z$ and $r$, then the following is valid for any Borel set $E$
$$\mu(E)\leq \Lambda_h(E).$$
\end{theorem}

For the converse relation we use a simple version of T.7.6.1.(a) in \cite{ Przytycki}. Here, $b(1)$ is
the Besicovitch covering number corresponding to the line (one can take $b(1)=5$).
\begin{theorem}  Assume that $\mu$ is a Borel probability measure on $\mathbb{R}$ and $A$ is a bounded
Borel subset of $\mathbb{R}$. If there exists a constant $C$ such that
$$ h(r)\leq C\cdot \mu(B(x,r))$$
for all $x \in A$ and $r>0,$ then for any Borel set $E\subset A$
$$\Lambda_h(E)\leq b(1)\,C \cdot \mu(E).$$
\end{theorem}

\vspace{7mm}

The set $K(\gamma)$ is weakly equilibrium in the following sense. Given $s\in\mathbb{N},$ we uniformly distribute the mass $2^{-s}$ on each $I_{j,s}$ for $1\leq j\leq 2^{-s}$. Let us denote by $\lambda_s$ the normalized in this sense Lebesgue measure on $E_s,$ so $d\lambda_s=(2^s l_{j,s})^{-1} dt$ on $I_{j,s}$.

\begin{theorem} (\cite{Goncharov}\label{yyyuuu},T.4) Suppose $K(\gamma)$ is not polar. Then $\lambda_s$ is weak star convergent to the equilibrium measure $\mu_{K(\gamma)}$.
\end{theorem}
\begin{corollary} \label{e}  Suppose $K(\gamma)$ is not polar. Then $\mu_{K(\gamma)}(I_{j,s})=2^{-s}$
for each $s\in \mathbb{N}$ and $1\leq j \leq 2^s.$
\end{corollary}
\begin{proof} Indeed, the characteristic function $\chi_{I_{j,s}}$ is continuous on $E_n$ for $n \geq s,$
 where $E_n$ is given in the construction of $K(\gamma)$.
Therefore, $\mu_{K(\gamma)} (I_{j,s})= \int\, \chi_{I_{j,s}}\, d\mu_{K(\gamma)} =
\lim_{n \to \infty} \int \, \chi_{I_{j,s}}\,d\lambda_n = 2^{-s}.$
\end{proof}

In our main theorem and below, by $\Lambda_h$ we mean restricted to the compact set $K(\gamma)$ the Hausdorff measure corresponding to the constructed function $h$.

\begin{theorem} \label{aaauuu}Let $\gamma$ satisfy (2.2) and $K(\gamma)$ be non polar.  Then measures
$\mu_{K(\gamma)}$ and $ \Lambda_h$ are mutually absolutely continuous.
\end{theorem}
\begin{proof}
Let us fix any open interval $I$ of length $2r$ and show that
\begin{equation}
\mu_{K(\gamma)}(I) \leq 8\,h(r).
\end{equation}
Then,  by Theorem 3.1, $\mu_{K(\gamma)}(E) \leq 8 \, \Lambda_h(E)$ for any Borel set $E$ and
$\mu_{K(\gamma)} \ll \Lambda_h.$

First suppose that the endpoints of $I$ do not belong to $K(\gamma)$. Then there exists $I^{\prime}=[a,b] \subset I$ which contains all points in $K(\gamma)\cap I$. Let us take, as above, minimal $n$ and $q$ such that both
$a$ and $b$  are the endpoints of $n-$th level and $I^{\prime}$ contains at least one basic interval $I_{j,q}.$
All points in $K(\gamma)\cap I$ can be covered by $4$ adjacent intervals of the level $q$ and this cover $G$ contains $4\cdot 2^{s-q}$ intervals of the level $s$ for $s\geq q$. Hence $ \int \chi_G\,\,d\lambda_s=4\cdot 2^{-q}.$
The characteristic function $\chi_{I^{\prime}}$ is continuous on $E_s$ for $s \geq n.$ By Theorem \ref{yyyuuu},

$$ 4\cdot 2^{-q} \geq \lim_{s\rightarrow\infty}\int \, \chi_{I^{\prime}}\,d\lambda_s=
\int\, \chi_{I^{\prime}}\, d\mu_{K(\gamma)} =\mu_{K(\gamma)}(I).$$

On the other hand, $I$ contains some  basic interval $I_{j,q}.$ Therefore  $2r>l_{j,q}$ and, by (2.3),
$$h(2r)\geq h(l_{j,q})\geq h(\delta_{q})=2^{-q}.$$
Combining these inequalities with (2.4) gives (3.1):
 $$8\,h(r)>4\,h(2r)\geq 4\cdot 2^{-q}\geq \mu_{K(\gamma)}(I).$$

Now let us consider the case when at least one of the endpoints of $I=(z-r,z+r)$ is contained in $K(\gamma)$.
Since the set is totally disconnected, we can take two real null sequences $(\alpha_n)_{n=1}^{\infty}$ and $(\beta_n)_{n=1}^{\infty}$ such that the endpoints of $I_n= (z-r-\alpha_n, z+r+\beta_n)$ do not belong to $K(\gamma)$ for each $n$.
Arguing as above, we see that
$$\mu_{K(\gamma)}(I) \leq \mu_{K(\gamma)}(I_n)\leq 8\,h(r+\varepsilon_n) \leq 8(1+\varepsilon_n/r)\,h(r),$$
where $\varepsilon_n=\max\{ \alpha_n,\beta_n\}.$ Since $\varepsilon_n \to 0$ as $n \to \infty,$ we have
(3.1) for this case as well.\\

We proceed to show that $ \Lambda_h  \ll \mu_{K(\gamma)}.$ Let us fix $x \in K(\gamma)$ and $r>0.$ In order to use
Theorem 3.2, let us show that
\begin{equation}
h(r) \leq 2M  \cdot \mu_{K(\gamma)}(I),
\end{equation}
where $I=(x-r, x+r).$ Clearly, it is enough to consider only $r<1.$
Let us fix two consecutive basic intervals containing our point: $x\in I_{i,\,s} \subset I_{j,\,s-1}$ with
$ l_{i,\,s} \leq r < l_{j,\,s-1}.$ Then $I \supset I_{i,\,s}$ and $\mu_{K(\gamma)}(I) \geq 2^{-s},$
by Corollary 3.4. On the other hand, by (2.3) and (2.4),
$$h(r) < h(l_{j,\,s-1}) < h(M\,\delta_{s-1}) < M \,h(\delta_{s-1})= M\,2^{-s+1}.$$
This gives (3.2) and completes the proof.
\end{proof}
\vspace{5mm}

\begin{example} The sequence $\gamma$ with  $\gamma_s=\exp{(-8s+4)}$ for $s\in\mathbb{N}$ satisfies all desired
conditions. In particular, $Cap(K(\gamma)) = \exp{(-12)},$ so $K(\gamma)$ is not polar.
Here, $\delta_s=\exp{(-4 s^2)}$ and
$\eta(t)= s + \frac{-4s^2-\log{t}}{8s+4}$ for $\delta_s \leq t < \delta_{s-1}.$ The Hausdorff measure $\Lambda_h$ corresponding to the function $h=2^{-\eta(t)}$ and $\mu_{K(\gamma)}$ are mutually absolutely continuous.
\end{example}

\section{Regularity of $\mu_{K(\gamma)}$ and $\Lambda_{h}$ in Stahl-Totik sense}

One of active directions of the theory of general orthogonal polynomials is the exploration of
the case of non discrete measures that are singular with respect to the Lebesgue measure. Important class of 
{\it regular in Stahl-Totik sense} measures was introduced in \cite{Stahl} in the following way. Let
$\mu$ be a finite Borel measure with compact support $S_{\mu}$ on $\mathbb{C}$. Then we can uniquely 
define a sequence of orthonormal polynomials $p_n(\mu;z) = a_n z^n + \ldots$ with a positive leading coefficient
$a_n.$ By definition, $\mu \in {\bf Reg}$ if  $\lim_{n\rightarrow\infty} {a_n}^{-\frac{1}{n}} = Cap(S_{\mu})$.
One of sufficient conditions of regularity was suggested in \cite{Stahl} by means of the set
 $A_{\mu}=\{z\in S_{\mu}:\,  \limsup_{r\rightarrow 0^+} \frac{\log 1/\mu(\overline{B(z,r)})}{\log 1/r}<\infty\}.$
 
\begin{theorem}\label{llml} ( T.4.2.1 in \cite{Stahl}) If $Cap(A_{\mu}) = Cap(S_{\mu})$ then $\mu \in {\bf Reg}.$ 
\end{theorem}

Let us show that, in our case, $A_{\mu} = S_{\mu}$ for both measures  $\mu_{K(\gamma)}$ and $\Lambda_{h}.$

\begin{theorem} Let $K(\gamma)$ satisfy the conditions of Theorem \ref{aaauuu}. Then $\mu_{K(\gamma)}$ and $\Lambda_{h}$ are regular in Stahl-Totik sense.
\end{theorem}
\begin{proof} Since $\Lambda_{h}(E)\geq \mu_{K(\gamma)(E)} /8 $ for any Borel subset $E$ of $K(\gamma),$ we only check the equilibrium measure. Let $z\in K(\gamma)$ and $r>0$ be given. As in Theorem 3.5, fix $s$ such  that
$z \in I_{i,\,s} \subset I_{j,\,s-1}$ with $ l_{i,\,s} \leq r < l_{j,\,s-1}.$ 
Then $I_{i,\,s} \subset \overline{B(z,r)}.$ By Corollary 3.4, $\mu(\overline{B(z,r)})\geq \mu(I_{i,\,s}) =2^{-s}.$
On the other hand, $r< M\,\delta_{s-1} \leq M\,32^{1-s}$ as $\gamma_k\leq 1/32.$
Since $s \to \infty$ as $r\rightarrow 0^+,$ we see that 
$$\limsup_{r\rightarrow 0^+} \frac{\log 1/\mu(\overline{B(z,r)})}{\log 1/r}\leq 1/5,$$
which completes the proof.
\end{proof}

\bibliographystyle{amsplain}

\begin{thebibliography}{99}
\bibitem{Batakis} A. Batakis, \textit{Harmonic measure of some Cantor type sets}, Ann. Acad. Sc. Fenn. \textbf{21} (1996), 27-54.

\bibitem{Best} E. Best, \textit{A closed dimensionless linear set}, Proc. Edinburgh Mat. Soc. \textbf{6}(2) (1939), 105-108.

\bibitem{Garnett} J.B. Garnett, D.E. Marshall, \textit{Harmonic measure}, Cambridge University Press, 2005.

\bibitem{Goncharov} A. Goncharov, \textit{Weakly Equilibrium Cantor-type Sets}, Potential Analysis \textbf{40}(2) (2014), 143-161.

\bibitem{Makarov} N.G. Makarov, \textit{On the distortion of boundary sets under conformal mappings}, Proc. London Math. Soc. \textbf{51} (1985), 369-384.

\bibitem{Makarov Volberg} N.G. Makarov, A. Volberg, \textit{On the harmonic measure of discontinuous fractals}, LOMI Preprints, E-6-86, Steklov Mathematical Institute, Leningrad Department, 1986.

\bibitem{Nevanlinna} R. Nevanlinna, \textit{Analytic Functions}, Springer-Verlag, 1970.

\bibitem{Pommerenke} Ch. Pommerenke, \textit{On conformal mapping and linear measure}, J. Anal. Math. \textbf{46}
(1986), 231-238.

\bibitem{Przytycki} F. Przytycki, M. Urbanski, \textit{Conformal fractals: ergodic theory methods.} London Math. Soc. Lecture Note Series, 371. Cambridge University Press, Cambridge, 2010.

\bibitem{Ransford} T. Ransford, \textit{Potential theory in the complex plane}, Cambridge University Press, 1995.

\bibitem{Stahl} H. Stahl, V. Totik, \textit{General Orthogonal Polynomials}, Cambridge University Press, 1992.

\bibitem{Volberg} A. Volberg, \textit{On the dimension of harmonic measure of Cantor repellors}, Michigan Math. J. \textbf{40}(2) (1993), 239-258.

\bibitem{Zdunik} A. Zdunik, \textit{Harmonic measure on the Julia set for polynomial-like maps}, Inventiones Mathematicae \textbf{128}(2) (1997), 303-327.
\end{thebibliography}

\end{document}